\documentclass[reqno]{amsart}

\usepackage{amsmath,amssymb,amsfonts,amsthm}
\usepackage{microtype}
\usepackage[hidelinks]{hyperref}

\newtheorem{theorem}{Theorem}[section]
\newtheorem{lemma}[theorem]{Lemma}
\newtheorem{proposition}[theorem]{Proposition}

\newtheorem{remark}[theorem]{Remark}
\numberwithin{equation}{section}

\begin{document}

\title[Overdetermined problems for
sum Hessian equations]{Overdetermined problems for
sum Hessian equations in the space forms}

\author{Cheng Cheng}
\address{School of Mathematics and Statistics, Xinyang Normal University,
Xinyang 464000, P.R.China}
\email{chengcheng26@126.com}

\author{Yingbo Han}
\address{School of Mathematics and Statistics, Xinyang Normal University,
Xinyang 464000, P.R.China}
\email{yingbohan@163.com}


\begin{abstract}
We study overdetermined problems for normalized sum Hessian equations and
their quotients in the space forms. The equations and boundary conditions imply
admissibility and ellipticity. For quotient equations, we prove that the
domain is a geodesic ball if it is strictly star-shaped, or if the
Dirichlet value satisfies an upper bound. We also prove the
corresponding result for sum Hessian equations without star-shapedness.
The proofs use two auxiliary functions and a Rellich--Poho\v{z}aev identity.
\end{abstract}

\maketitle

\section{Introduction}\label{sec:introduction}

Overdetermined boundary conditions relate elliptic equations to the
geometry of their domains. Serrin \cite{Serrin} proved that a bounded
domain admitting a solution of
$$
 \Delta u=-1\quad\hbox{in }\Omega,\qquad
 u=0,\quad u_\nu=\hbox{constant}\quad\hbox{on }\partial\Omega
$$
must be a ball. His proof used the moving plane method. Weinberger
\cite{Weinberger} gave a different proof based on a gradient auxiliary
function and an integral identity.

Several authorities have extended these arguments to fully nonlinear
equations. Brandolini, Nitsch, Salani and Trombetti \cite{Brandolini}
proved a Serrin-type theorem for the $k$-Hessian equation in Euclidean
space. Ma \cite{Ma} used an auxiliary function for a capillarity problem
for the two-dimensional Monge--Amp\`ere equation. Philippin and Safoui
\cite{PhilippinSafoui} obtained maximum principles for combinations of
solutions and their gradients for fully nonlinear equations. More
recently, Gavitone, Masiello, Paoli and Poggesi \cite{GavitoneEtAl}
studied symmetry and quantitative stability for Euclidean Hessian
overdetermined problems and their relation to higher order mean
curvatures.

The corresponding problems in space forms require curvature terms in
both the equation and the auxiliary functions. Molzon \cite{Molzon}
and Kumaresan and Prajapat \cite{KumaresanPrajapat} studied symmetry
results on spheres and hyperbolic spaces. Qiu and Xia \cite{QiuXia}
used two auxiliary functions for a spherical overdetermined problem.
Ciraolo and Vezzoni \cite{CiraoloVezzoni} gave a proof of Serrin's theorem
in space forms by integral methods. Gao, Jia and Yan \cite{GaoJiaYan}
consider  overdetermined problems of Hessian equation in hyperbolic space.
Subsequently, Gao, Ma and Yang \cite{GaoMaYang} considered Hessian quotient equations
with constant Dirichlet condition in the space forms. Their
Theorem~1.1 proves symmetry under a star-shapedness assumption. Their
Theorem~1.4 consider the $k$-Hessian equation without this assumption,
subject to a bound on the Dirichlet value when $k\geq2$. Their proofs
use two auxiliary functions and a Rellich--Poho\v{z}aev identity. We
extend these arguments to adjacent sum Hessian operators. Related results are obtained by Gao-Jia-Zhang \cite{GaoJiaZhang}.

On the other hand, Li, Ren and Wang \cite{LiRenWang} studied curvature estimates
for convex solutions of Hessian-type equations defined by
linear combinations of elementary symmetric functions.
Liu and Ren \cite{LiuRen} established Pogorelov-type estimates
for sum Hessian equations and applied them to prove a rigidity
theorem for entire solutions. Li and Ren \cite{LiRen} obtained
further Pogorelov-type estimates for admissible solutions under
a lower bound on $\sigma_k(D^2u)$. They removed this lower bound
when $k=n$. Ren and Wang \cite{RenWangSumQuotient} studied
interior second derivative estimates for sum Hessian quotient
equations. These works concern a priori estimates and rigidity
of entire solutions.

Motivated by the above results, in this paper we extend the overdetermined problem to the setting of sum Hessian and sum Hessian quotient equations.

Let $(M^n(K),g)$ be the complete simply connected space form of
sectional curvature $K$, where $n\geq2$. Throughout this paper,
$\Omega\subset M^n(K)$ is a bounded connected domain with $C^2$
boundary, and $\nu$ is its outward unit normal. If $K>0$, we assume
that $\overline\Omega$ lies in an open hemisphere.

For a symmetric endomorphism $A$, let $\sigma_j(A)$ denote the $j$-th
elementary symmetric function of its eigenvalues. We write
\begin{equation}\label{eq:normalized}
 E_j(A)=\binom nj^{-1}\sigma_j(A),\qquad 0\leq j\leq n,
\end{equation}
where $E_0=1$, and set $E_j=0$ for $j<0$ or $j>n$. For a fixed
constant $\alpha\geq0$, define
\begin{equation}\label{eq:G-definition}
 G_j(A)=E_j(A)+\alpha E_{j-1}(A),\qquad 1\leq j\leq n,
 \qquad G_0=1.
\end{equation}
In particular, $G_j(I)=1+\alpha$ for every $j\geq1$.

%Sum Hessian equations have been studied by Liu and Ren \cite{LiuRen},
%who obtained Pogorelov estimates and an entire-solution rigidity
%theorem. Ren \cite{Ren} proved Newton--Maclaurin inequalities for
%the normalized adjacent sums in \eqref{eq:G-definition}.
%Hu, Ren and Wang \cite{HuRenWang} obtained related inequalities for
%longer linear combinations. In this paper, the normalization of each
%$E_j$ allows us to compare different orders with the same parameter
%$\alpha$.

For a function $u$, we use the symmetric tensor
\begin{equation}\label{eq:b-definition}
 b=\nabla^2u+Kug.
\end{equation}
The metric identifies $b$ with a symmetric endomorphism when we
evaluate $E_j$ or $G_j$. Let $c_1\in\mathbb R$ and $c_2>0$ be
constants. Our first problem is
\begin{equation}\label{eq:quotient-problem}
 \begin{cases}
 \displaystyle\frac{G_k(b)}{G_l(b)}=1&\text{in }\Omega,\\[5pt]
 u=Kc_1&\text{on }\partial\Omega,\\
 u_\nu=c_2&\text{on }\partial\Omega,\\
 Ku\geq0&\text{in }\Omega,
 \end{cases}
 \qquad 1\leq l<k\leq n.
\end{equation}
A solution is understood to satisfy $G_l(b)\ne0$ in $\Omega$. We obtain the following result.

\begin{theorem}\label{thm:star-quotient}
Let $1\leq l<k\leq n$, and let $u\in C^3(\overline\Omega)$ solve
\eqref{eq:quotient-problem}. Assume that $\Omega$ is strictly
star-shaped with respect to a point $o$. If $K>0$, assume also that
$\overline\Omega$ lies in the open hemisphere centered at $o$.
Then
 $\Omega$ is a geodesic ball and $u$ is radial with respect
to its center.
\end{theorem}

The following theorem replaces star-shapedness by a condition on the
Dirichlet value.

\begin{theorem}\label{thm:nonstar-quotient}
Let $1\leq l<k\leq n$, and let $u\in C^3(\overline\Omega)$ solve
\eqref{eq:quotient-problem}. If $K^2c_1\leq1$, then
 $\Omega$ is a geodesic ball and $u$ is radial with respect
to its center.
\end{theorem}

We also consider the sum Hessian problem
\begin{equation}\label{eq:sum-problem}
 \begin{cases}
 G_k(b)=1+\alpha&\text{in }\Omega,\\
 u=Kc_1&\text{on }\partial\Omega,\\
 u_\nu=c_2&\text{on }\partial\Omega,\\
 Ku\geq0&\text{in }\Omega,
 \end{cases}
 \qquad 1\leq k\leq n.
\end{equation}

We obtain the following result.

\begin{theorem}\label{thm:sum}
Let $1\leq k\leq n$, and let $u\in C^3(\overline\Omega)$ solve
\eqref{eq:sum-problem}. If $k\geq2$, assume that $K^2c_1\leq1$.
Then
 $\Omega$ is a geodesic ball and $u$ is radial with respect
to its center.
\end{theorem}

\begin{remark}
\begin{enumerate}
\item When $k=1$, the additional restriction on $c_1$ is not required.
\item When $\alpha=0$, Theorems~\ref{thm:star-quotient} and \ref{thm:sum}
recover the corresponding results of Gao, Ma and Yang
\cite[Theorems~1.1 and 1.4]{GaoMaYang}.
\item When $\alpha=0$, Theorem~\ref{thm:nonstar-quotient} gives a non-star-shaped
Hessian quotient result under $K^2c_1\leq1$, whereas the quotient
theorem in \cite{GaoMaYang} assumes star-shapedness.
For $K=0$, the condition $K^2c_1\leq1$ holds automatically.
\end{enumerate}
\end{remark}

Section~\ref{sec:preliminaries} contains the space-form identities and
the symmetric-function inequalities used below.
Section~\ref{sec:ellipticity} proves admissibility, ellipticity and
the algebraic estimates for the quotient.
In Section~\ref{sec:auxiliary}, we study the two auxiliary functions
and prove Theorem~\ref{thm:star-quotient}.
Section~\ref{sec:integral} establishes the integral identity and
proves Theorems~\ref{thm:nonstar-quotient} and \ref{thm:sum}.

\section{Preliminaries}\label{sec:preliminaries}

We collect the identities for space forms and normalized elementary
symmetric functions. We also recall the Newton--Maclaurin
inequalities and the divergence identity for the tensor
$b=\nabla^2u+Kug$ ( see \cite{GaoMaYang} for details). All tensor calculations use local orthonormal
frames, and repeated indices are summed.

Fix $o\in M^n(K)$ and put $r=d(o,\cdot)$. If $K>0$, choose $o$ so
that $\overline\Omega$ lies in the open hemisphere centered at $o$.
Set
$$
 f(r)=
 \begin{cases}
 r,&K=0,\\
 \dfrac{\sin(\sqrt K r)}{\sqrt K},&K>0,\\[4pt]
 \dfrac{\sinh(\sqrt{-K}r)}{\sqrt{-K}},&K<0,
 \end{cases}
 \qquad V(r)=f'(r).
$$
We use the potential
$$
 \Phi(r)=
 \begin{cases}
 \dfrac12r^2,&K=0,\\[3pt]
 -\dfrac{\cos(\sqrt K r)}K,&K>0,\\[5pt]
 \dfrac{\cosh(\sqrt{-K}r)}{-K},&K<0.
 \end{cases}
$$
These functions are smooth at $o$. The warped product expression
for the space-form metric gives
\begin{equation}\label{eq:space-form-identities}
 \nabla\Phi=f(r)\partial_r,\qquad
 \nabla^2\Phi=Vg,\qquad
 \nabla V=-K\nabla\Phi,\qquad
 \nabla^2V=-KVg.
\end{equation}
These are the identities used in
\cite[Proposition~2.1]{GaoMaYang}. Our hemisphere assumption implies
$V>0$ on $\overline\Omega$. Strict star-shapedness with respect to
$o$ means
\begin{equation}\label{eq:star-shaped}
 \Phi_\nu=f(r)\langle\partial_r,\nu\rangle>0
 \quad\text{on }\partial\Omega.
\end{equation}

For any $v\in C^2(\overline\Omega)$, the product rule and
$\Delta V=-nKV$ yield
\begin{equation}\label{eq:positive-transform}
 \Delta\!\left(\frac vV\right)
 +\frac2V\left\langle\nabla V,\nabla\!\left(\frac vV\right)\right\rangle
 =\frac{\Delta v+nKv}{V}.
\end{equation}
The operator on the left has no zeroth-order term. We shall apply
the maximum principle and the Hopf boundary point lemma in
\cite[Chapter~3]{GilbargTrudinger}. The $C^2$ boundary satisfies
the interior sphere condition.

\subsection{Hessian and sum Hessian operator}

For a symmetric matrix $A=(a_{ij})$, write
$$
 E_m^{ij}(A)=\frac{\partial E_m}{\partial a_{ij}}(A),
 \qquad
 G_m^{ij}(A)=\frac{\partial G_m}{\partial a_{ij}}(A).
$$
The normalized Newton identities are
\begin{align}
 E_m^{ij}a_{ij}&=mE_m,\label{eq:Euler-E}\\
 E_m^{ij}\delta_{ij}&=mE_{m-1},\label{eq:trace-E}\\
 E_m^{ij}a_{pi}a_{pj}
 &=nE_1E_m-(n-m)E_{m+1}.
 \label{eq:square-E}
\end{align}
They follow from the corresponding identities for $\sigma_m$ in
\cite[Proposition~1.2]{Reilly}, after division by $\binom nm$.
Since $G_m^{ij}=E_m^{ij}+\alpha E_{m-1}^{ij}$, we obtain
\begin{align}
 G_m^{ij}a_{ij}
 &=mG_m-\alpha E_{m-1},\label{eq:Euler-G}\\
 G_m^{ij}\delta_{ij}
 &=mG_{m-1}-\alpha E_{m-2},\label{eq:trace-G}\\
 G_m^{ij}a_{pi}a_{pj}
 &=nE_1G_m-(n-m)G_{m+1}-\alpha E_m.
 \label{eq:square-G}
\end{align}
When $m=n$, the term $(n-m)G_{m+1}$ is absent.

We use the G{\aa}rding cones
$$
 \Gamma_m=\{A:E_j(A)>0,\ 1\leq j\leq m\}
$$
and the open sets
\begin{equation}\label{eq:admissible-set}
 \widetilde\Gamma_m=\{A:G_j(A)>0,\ 1\leq j\leq m\}.
\end{equation}
Newton's inequality
\begin{equation}\label{eq:Newton}
 E_j^2\geq E_{j-1}E_{j+1},\qquad 1\leq j\leq n-1,
\end{equation}
holds for every real eigenvalue vector; see \cite[Section~1]{Ren}.
On $\Gamma_m$, the Newton--Maclaurin inequalities give
\begin{equation}\label{eq:Maclaurin}
 E_1\geq E_2^{1/2}\geq\cdots\geq E_m^{1/m},
 \qquad
 \frac{E_1}{E_0}\geq\frac{E_2}{E_1}
 \geq\cdots\geq\frac{E_m}{E_{m-1}}.
\end{equation}
Equality between two distinct orders in either chain forces all
eigenvalues to be equal. We use this equality statement only when the
quantities involved are positive; see
\cite[Proposition~2.4]{GaoMaYang}.

Ren's inequality \cite[Theorem~1.2]{Ren} states that
\begin{equation}\label{eq:Ren}
 G_j^2\geq G_{j-1}G_{j+1},\qquad 2\leq j\leq n-1.
\end{equation}
It holds for all real eigenvalues and all real $\alpha$. At the
first index, our convention $G_0=1$ gives
\begin{equation}\label{eq:first-Newton}
 G_1^2-G_2=E_1^2-E_2+\alpha G_1\geq0
 \quad\text{if }G_1\geq0.
\end{equation}
Here we use $\alpha\geq0$. Thus, on $\widetilde\Gamma_m$,
\begin{equation}\label{eq:G-ratios}
 \frac{G_1}{G_0}\geq\frac{G_2}{G_1}
 \geq\cdots\geq\frac{G_m}{G_{m-1}}>0.
\end{equation}
For $2\leq j\leq m$, multiplying the first $j-1$ ratios shows that
$$
 G_{j-1}\geq
 \left(\frac{G_j}{G_{j-1}}\right)^{j-1}.
$$
Consequently,
\begin{equation}\label{eq:G-Maclaurin}
 G_1\geq G_2^{1/2}\geq\cdots\geq G_m^{1/m}
 \quad\text{on }\widetilde\Gamma_m.
\end{equation}
This argument also covers $m=n$. We shall pass to limits from within
$\widetilde\Gamma_m$ when using \eqref{eq:G-Maclaurin} at a limit point.

We finish this section with the differential identities for $b$.
The Ricci identity in a space form gives
$$
 \begin{aligned}
 \nabla_m b_{ij}
 &=\nabla_m u_{ij}+Ku_m\delta_{ij}\\
 &=\nabla_j u_{im}
   +K(u_j\delta_{im}-u_m\delta_{ij})+Ku_m\delta_{ij}\\
 &=\nabla_jb_{im}.
 \end{aligned}
$$
Thus $b$ is a Codazzi tensor:
\begin{equation}\label{eq:Codazzi}
 \nabla_m b_{ij}=\nabla_jb_{im}.
\end{equation}
Its Newton tensors satisfy
\begin{equation}\label{eq:Newton-divfree}
 \nabla_jE_m^{ij}(b)=0,\qquad
 \nabla_jG_m^{ij}(b)=0.
\end{equation}
For completeness, the normalized Newton recurrence is
$$
 E_m^{ij}(b)=\frac{m}{n-m+1}
 \left(E_{m-1}(b)\delta_{ij}-E_{m-1}^{lj}(b)b_{il}\right).
$$
For $m=1$, this is $E_1^{ij}=\delta_{ij}/n$. If the first identity
in \eqref{eq:Newton-divfree} holds for $m-1$, differentiation of the
recurrence gives
$$
 \begin{aligned}
 \nabla_jE_m^{ij}
 &=\frac{m}{n-m+1}
 \left(\nabla_iE_{m-1}
       -E_{m-1}^{lj}\nabla_jb_{il}\right)\\
 &=\frac{m}{n-m+1}
 \left(\nabla_iE_{m-1}
       -E_{m-1}^{lj}\nabla_ib_{lj}\right)=0.
 \end{aligned}
$$
The second identity follows by linearity. This is the Codazzi
Newton-tensor argument in \cite[Proposition~2.1]{Reilly}; see also
\cite[Proposition~2.3]{DellaPietraGavitoneXia}.
The recurrence also shows that $E_m^{ij}$ and $G_m^{ij}$, viewed
as endomorphisms, are polynomials in $b$. In particular, they commute
with $b$.

\section{Admissibility and ellipticity}\label{sec:ellipticity}

We first show that the equation and boundary conditions place $b$ in
$\widetilde\Gamma_k$ up to the boundary. We then prove ellipticity
of the sum and quotient operators. The last lemma gives the
inequalities needed for the auxiliary functions.

\begin{lemma}\label{lem:cone-inclusion}
If $k\geq2$ and $A\in\widetilde\Gamma_k$, then $A\in\Gamma_{k-1}$.
\end{lemma}

\begin{proof}
We prove $E_j(A)>0$ successively for $1\leq j\leq k-1$.
Assume that $E_{j-1}>0$, starting with $E_0=1$.
If $E_j\leq0$, Newton's inequality \eqref{eq:Newton} gives
$$
 \begin{aligned}
 G_{j+1}
 &=E_{j+1}+\alpha E_j\\
 &\leq\frac{E_j^2}{E_{j-1}}+\alpha E_j
 =\frac{E_j}{E_{j-1}}G_j\leq0.
 \end{aligned}
$$
This contradicts $G_{j+1}>0$. The induction proves the assertion.
\end{proof}

\begin{lemma}\label{lem:automatic-cone}
Let $u\in C^2(\overline\Omega)$ solve either
\eqref{eq:quotient-problem} or \eqref{eq:sum-problem}.
Then $b\in\widetilde\Gamma_k$ on $\overline\Omega$.
For the quotient problem, one also has
\begin{equation}\label{eq:G-lower-bound}
 G_k(b)=G_l(b)\geq1\quad\text{on }\overline\Omega.
\end{equation}
\end{lemma}

\begin{proof}
The function $u$ attains its minimum on $\overline\Omega$.
This minimum cannot occur at a boundary point $y$. Indeed, the
curve $\exp_y(-t\nu)$ lies in $\Omega$ for small $t>0$, and
$$
 \left.\frac{d}{dt}u(\exp_y(-t\nu))\right|_{t=0}=-c_2<0.
$$
Hence $u$ has an interior minimum at some $x_0$. At this point, we can see that
$$
 b(x_0)=\nabla^2u(x_0)+Ku(x_0)g\geq0.
$$
All elementary symmetric functions of $b(x_0)$ are nonnegative.
For the quotient problem, $G_l(b(x_0))\ne0$ and $G_k=G_l$,
so $G_k(b(x_0))>0$. For the sum problem,
$G_k(b(x_0))=1+\alpha>0$.

If $\alpha=0$, positivity of $G_k=E_k$ at a nonnegative
eigenvalue vector implies positivity of $E_j$ for $j\leq k$.
If $\alpha>0$, positivity of $G_k$ implies that at least $k-1$
eigenvalues are positive when $k\geq2$. Thus $G_j(b(x_0))>0$
for $j<k$ as well. The case $k=1$ follows directly from
$G_1(b(x_0))>0$. We have proved that
$b(x_0)\in\widetilde\Gamma_k$.

Let
$$
 \mathcal U=\{x\in\overline\Omega:b(x)\in\widetilde\Gamma_k\}.
$$
This set is nonempty and relatively open. We prove that it is
relatively closed.

First consider the quotient problem. On $\mathcal U\cap\Omega$,
the equation and \eqref{eq:G-ratios} give
$$
 1=\frac{G_k}{G_l}
 =\frac{G_{l+1}}{G_l}\cdots\frac{G_k}{G_{k-1}},
 \qquad
 \frac{G_1}{G_0}\geq\cdots\geq\frac{G_k}{G_{k-1}}>0.
$$
The first factor in the product is at least one. Each of the
preceding ratios is therefore also at least one. It follows that
$$
 G_l=\frac{G_1}{G_0}\cdots\frac{G_l}{G_{l-1}}\geq1.
$$
Continuity gives $G_k=G_l\geq1$ on all of $\mathcal U$.
For the sum problem, $G_k=1+\alpha$ on $\overline\Omega$ by
continuity.

Now take $x_m\in\mathcal U$ with $x_m\to x\in\overline\Omega$.
The positive form of \eqref{eq:G-Maclaurin} gives
$$
 G_j(b(x_m))\geq G_k(b(x_m))^{j/k},
 \qquad 1\leq j\leq k.
$$
In the quotient case the right-hand side is at least one.
In the sum case it equals $(1+\alpha)^{j/k}>0$.
Passing to the limit proves $G_j(b(x))>0$ for every $j\leq k$.
Thus $\mathcal U$ is relatively closed.
Since $\overline\Omega$ is connected, $\mathcal U=\overline\Omega$.
The preceding estimate also proves \eqref{eq:G-lower-bound}.
\end{proof}

\begin{lemma}\label{lem:individual-ellipticity}
If $A\in\widetilde\Gamma_k$, then $(G_j^{pq}(A))$ is positive
definite for every $1\leq j\leq k$.
\end{lemma}

\begin{proof}
The assertion for $j=1$ follows from $G_1^{pq}=\delta_{pq}/n$.
Let $j\geq2$, and let $\lambda=(\lambda_1,\ldots,\lambda_n)$
be the eigenvalues of $A$. Set
$$
 \beta_j=\alpha\frac{n-j+1}{j}.
$$
The elementary symmetric function in $n+1$ variables satisfies
\begin{equation}\label{eq:augmented-G}
 \sigma_j(\lambda_1,\ldots,\lambda_n,\beta_j)
 =\sigma_j(\lambda)+\beta_j\sigma_{j-1}(\lambda)
 =\binom njG_j(\lambda)>0.
\end{equation}
It follows from  Lemma~\ref{lem:cone-inclusion} that $A\in\Gamma_{k-1}$.
For $m<j$, we therefore have
$$
 \sigma_m(\lambda_1,\ldots,\lambda_n,\beta_j)
 =\sigma_m(\lambda)+\beta_j\sigma_{m-1}(\lambda)>0.
$$
Thus $(\lambda,\beta_j)\in\Gamma_j$ in dimension $n+1$.

The elementary symmetric function $\sigma_j$ is elliptic on
$\Gamma_j$; see \cite{CNS} and
\cite[Proposition~2.3]{GaoMaYang}. Hence, for every $i$,
$$
 \frac{\partial G_j}{\partial\lambda_i}
 =\binom nj^{-1}
 \left(\sigma_{j-1}(\lambda\mid i)
       +\beta_j\sigma_{j-2}(\lambda\mid i)\right)>0.
$$
Here $\lambda\mid i$ means that the $i$-th entry is omitted.
In an eigenframe of $A$, these derivatives are the eigenvalues
of $(G_j^{pq})$. The first-variation formula for a symmetric
polynomial also holds at repeated eigenvalues. Therefore
$(G_j^{pq})$ is positive definite.
\end{proof}

\begin{proposition}\label{prop:quotient-ellipticity}
Let $1\leq l<k\leq n$. The operator
$$
 F(A)=\frac{G_k(A)}{G_l(A)}
$$
is elliptic on $\widetilde\Gamma_k$.
\end{proposition}

\begin{proof}
Fix an eigenvalue vector $\lambda$ in $\widetilde\Gamma_k$
and an index $i$. By Lemma~\ref{lem:individual-ellipticity},
$$
 a_j=\frac{\partial G_j}{\partial\lambda_i}(\lambda)>0,
 \qquad 1\leq j\leq k.
$$
Each $G_j$ is affine in $\lambda_i$. Consequently,
$$
 g_j(t)=G_j(\lambda+te_i)=G_j(\lambda)+a_jt.
$$
Set
$$
 t_0=\max_{1\leq j\leq k}
 \left(-\frac{G_j(\lambda)}{a_j}\right)<0.
$$
Then $g_j(t)>0$ for $t>t_0$, while $g_j(t_0)\geq0$ for all
$j\leq k$, and at least one of these values vanishes.

Let $r$ be the first index for which $g_r(t_0)=0$.
We claim that $r=k$. Suppose that $r<k$.
If $r=1$, \eqref{eq:first-Newton} and continuity give
$g_2(t_0)=0$. Applying the same inequality at $t_0+s$,
where $s>0$, gives
$$
 a_1^2s^2\geq a_2s.
$$
Division by $s$ and passage to the limit $s\downarrow0$
contradict $a_2>0$.

If $r\geq2$, then $g_{r-1}(t_0)>0$. Ren's inequality
\eqref{eq:Ren}, first at $t>t_0$ and then by continuity,
gives $g_{r+1}(t_0)=0$. At $t=t_0+s$, it reads
$$
 a_r^2s^2\geq
 \bigl(g_{r-1}(t_0)+a_{r-1}s\bigr)a_{r+1}s.
$$
After division by $s$ and passage to the limit, we obtain
$0\geq g_{r-1}(t_0)a_{r+1}>0$, a contradiction.
Thus $r=k$, and
$$
 g_k(t_0)=0,\qquad g_l(t_0)>0.
$$

The affine formulas now yield
$$
 a_kG_l(\lambda)-a_lG_k(\lambda)=a_kg_l(t_0)>0.
$$
Therefore
\begin{equation}\label{eq:quotient-eigenvalue-derivative}
 \frac{\partial F}{\partial\lambda_i}
 =\frac{a_kG_l-a_lG_k}{G_l^2}>0.
\end{equation}
This holds for every $i$. The spectral first-variation formula
then shows that $(F^{pq})$ is positive definite.
\end{proof}

For a solution of either boundary value problem, the relevant
linearization is uniformly elliptic on $\overline\Omega$.
Indeed, Lemma~\ref{lem:automatic-cone} places the eigenvalue
vectors in a compact subset of $\widetilde\Gamma_k$, and the
preceding results give strict positivity there. Their continuous
eigenvalues therefore have a positive lower bound.

\begin{lemma}\label{lem:quotient-algebra}
Let $A\in\widetilde\Gamma_k$, where $1\leq l<k\leq n$, and
assume $G_k(A)=G_l(A)=G$. Then
\begin{equation}\label{eq:quotient-neighbours}
 G\geq1,\qquad
 G_{k-1}\geq G,\qquad
 G_{l-1}\leq G,\qquad
 G_{l+1}\geq G.
\end{equation}
If $k<n$, one also has $G_{k+1}\leq G$. Moreover,
\begin{equation}\label{eq:quotient-E}
 E_1\geq1,\qquad
 E_{k-1}-E_{l-1}\geq0,\qquad
 E_l-E_k=\alpha(E_{k-1}-E_{l-1}).
\end{equation}
If $E_{k-1}=E_{l-1}$, then $A=I$.
\end{lemma}

\begin{proof}
We suppress the argument $A$ in this proof. The ratios in
\eqref{eq:G-ratios} are non-increasing, and
$$
 \frac{G_{l+1}}{G_l}\cdots\frac{G_k}{G_{k-1}}=1.
$$
Their first factor is at least one and their last factor is at
most one. Hence
$$
 \frac{G_l}{G_{l-1}}\geq
 \frac{G_{l+1}}{G_l}\geq1
 \geq\frac{G_k}{G_{k-1}}.
$$
These inequalities give all the comparisons in
\eqref{eq:quotient-neighbours}. In particular, the first $l$
ratios are at least one, so their product $G_l$ is at least one.
If $k<n$, Ren's inequality at index $k$ gives
$$
 G_{k+1}\leq\frac{G_k^2}{G_{k-1}}\leq G.
$$
This estimate does not require $G_{k+1}>0$.

Since $l\geq1$, the ratio $G_2/G_1$ precedes or equals
$G_{l+1}/G_l$. By Newton's inequality and $G_1>0$,
$$
 1\leq\frac{G_2}{G_1}
 =\frac{E_2+\alpha E_1}{E_1+\alpha}
 \leq\frac{E_1^2+\alpha E_1}{E_1+\alpha}=E_1.
$$

Next, the equality $G_k=G_l$ gives
\begin{equation}\label{eq:E-comparison}
 G(E_{k-1}-E_{l-1})
 =E_{k-1}E_l-E_{l-1}E_k.
\end{equation}
Lemma~\ref{lem:cone-inclusion} gives $E_j>0$ for
$j\leq k-1$. Thus the right-hand side of
\eqref{eq:E-comparison} is positive if $E_k\leq0$.
If $E_k>0$, then $A\in\Gamma_k$, and
\eqref{eq:Maclaurin} gives
$$
 \frac{E_k}{E_{k-1}}\leq\frac{E_l}{E_{l-1}}.
$$
The right-hand side of \eqref{eq:E-comparison} is again
nonnegative. This proves $E_{k-1}\geq E_{l-1}$.
The last identity in \eqref{eq:quotient-E} follows by
subtracting $G_k=G_l$.

Suppose that $E_{k-1}=E_{l-1}$. The case $E_k\leq0$
has been excluded. Equality therefore holds between two
distinct positive Newton quotients in $\Gamma_k$.
The equality case in \eqref{eq:Maclaurin} shows that all
eigenvalues equal a number $a>0$. The equation becomes
$$
 a^{l-1}(a+\alpha)(a^{k-l}-1)=0.
$$
Since $a+\alpha>0$, we obtain $a=1$.
\end{proof}

\section{Auxiliary functions and star-shaped domains}\label{sec:auxiliary}

We study the two auxiliary functions used for the quotient equation.
The first gives a gradient estimate. The second detects equality in
the algebraic inequalities. We then prove
Theorem~\ref{thm:star-quotient}.

Throughout the auxiliary-function calculations, $u$ solves
\eqref{eq:quotient-problem}. We write
$$
 F=\frac{G_k}{G_l},\qquad
 F^{ij}=\frac{\partial F}{\partial b_{ij}},\qquad
 G=G_k(b)=G_l(b).
$$
The denominator $G$ is positive by Lemma~\ref{lem:automatic-cone}.
Define
\begin{equation}\label{eq:P-definition}
 P=|\nabla u|^2+Ku^2-2u,\qquad
 P_0=c_2^2+K^3c_1^2-2Kc_1.
\end{equation}
This is the auxiliary function used for Hessian equations in
\cite{GaoJiaYan,GaoMaYang}. Since $u$ is constant on the boundary,
$\nabla u=c_2\nu$ there, and
\begin{equation}\label{eq:P-boundary}
 P=P_0\quad\text{on }\partial\Omega.
\end{equation}

\begin{lemma}\label{lem:quotient-P}
For a solution of \eqref{eq:quotient-problem},
$$
 F^{ij}\nabla_{ij}^2P\geq0\quad\text{in }\Omega.
$$
Consequently, $P\leq P_0$ on $\overline\Omega$.
Either $P\equiv P_0$, or $P<P_0$ throughout $\Omega$.
\end{lemma}

\begin{proof}
Differentiation gives
\begin{equation}\label{eq:P-gradient}
 \frac12\nabla_iP
 =u_mu_{mi}+Kuu_i-u_i=b_{mi}u_m-u_i.
\end{equation}
Using the Codazzi identity \eqref{eq:Codazzi}, we obtain
\begin{align}
 \frac12\nabla_{ij}^2P
 &=(\nabla_jb_{mi})u_m+b_{mi}u_{mj}-u_{ij}\notag\\
 &=(\nabla_mb_{ij})u_m
   +b_{mi}(b_{mj}-Ku\delta_{mj})
   -(b_{ij}-Ku\delta_{ij})\notag\\
 &=(\nabla_mb_{ij})u_m+b_{mi}b_{mj}
   -(Ku+1)b_{ij}+Ku\delta_{ij}.
 \label{eq:P-Hessian}
\end{align}
The equation $F=1$ implies
$F^{ij}\nabla_mb_{ij}=0$. On the solution,
$$
 F^{ij}=\frac{G_k^{ij}-G_l^{ij}}G.
$$
The identities \eqref{eq:Euler-G}--\eqref{eq:square-G} therefore give
\begin{align}
 F^{ij}b_{ij}
 &=k-l-\alpha\frac{E_{k-1}-E_{l-1}}G,
 \label{eq:Fb}\\
 F^{ij}\delta_{ij}
 &=\frac{(k-1)G_{k-1}-(l-1)G_{l-1}
          +E_{k-1}-E_{l-1}}G,
 \label{eq:Ftrace}\\
 F^{ij}b_{mi}b_{mj}
 &=\frac{-(n-k)G_{k+1}+(n-l)G_{l+1}
          -\alpha(E_k-E_l)}G.
 \label{eq:Fsquare}
\end{align}
The term containing $G_{k+1}$ is absent when $k=n$.

Lemma~\ref{lem:quotient-algebra} yields
$$
 \begin{aligned}
 F^{ij}\delta_{ij}
 &\geq k-l+\frac{E_{k-1}-E_{l-1}}G,\\
 F^{ij}b_{mi}b_{mj}
 &\geq k-l+\alpha^2\frac{E_{k-1}-E_{l-1}}G.
 \end{aligned}
$$
Contracting \eqref{eq:P-Hessian}, and using $Ku\geq0$ in
the trace term, gives
\begin{align}
 \frac12F^{ij}\nabla_{ij}^2P
 &\geq k-l+\alpha^2\frac{E_{k-1}-E_{l-1}}G\notag\\
 &\quad-(Ku+1)
    \left(k-l-\alpha\frac{E_{k-1}-E_{l-1}}G\right)\notag\\
 &\quad+Ku\left(k-l+\frac{E_{k-1}-E_{l-1}}G\right)\notag\\
 &=\frac{(1+\alpha)(\alpha+Ku)}G
       (E_{k-1}-E_{l-1})\geq0.
 \label{eq:quotient-LP}
\end{align}
The operator is uniformly elliptic on $\overline\Omega$.
The maximum principle, its strong form, and
\eqref{eq:P-boundary} give the assertions;
see \cite[Chapter~3]{GilbargTrudinger}.
\end{proof}

The second function is
\begin{equation}\label{eq:Ptilde-definition}
 \widetilde P=-\langle\nabla u,\nabla\Phi\rangle+uV+\Phi.
\end{equation}
It is the space-form auxiliary function of
\cite[Lemma~4.2]{GaoMaYang}, based on the spherical construction
of Qiu and Xia \cite{QiuXia}.

\begin{lemma}\label{lem:quotient-Ptilde}
For a solution of \eqref{eq:quotient-problem},
$$
 F^{ij}\nabla_{ij}^2\widetilde P\geq0\quad\text{in }\Omega.
$$
If $\widetilde P$ is constant, then $b=g$.
\end{lemma}

\begin{proof}
The identities \eqref{eq:space-form-identities} give
\begin{align}
 \nabla_i\widetilde P
 &=-u_{mi}\Phi_m-u_m\Phi_{mi}+u_iV+uV_i+\Phi_i\notag\\
 &=-u_{mi}\Phi_m-Ku\Phi_i+\Phi_i
   =(\delta_{im}-b_{im})\Phi_m.
 \label{eq:Ptilde-gradient}
\end{align}
Differentiating once more and using \eqref{eq:Codazzi}, we find
$$
 \nabla_{ij}^2\widetilde P
 =-\Phi_m\nabla_mb_{ij}+V(\delta_{ij}-b_{ij}).
$$
After contraction, the differentiated equation removes the
first term. Hence
\begin{equation}\label{eq:Ptilde-Hessian}
 F^{ij}\nabla_{ij}^2\widetilde P
 =V(F^{ij}\delta_{ij}-F^{ij}b_{ij}).
\end{equation}
Equations \eqref{eq:Fb}, \eqref{eq:Ftrace} and
\eqref{eq:quotient-neighbours} show that
\begin{equation}\label{eq:Ftrace-gap}
 F^{ij}\delta_{ij}-F^{ij}b_{ij}
 \geq(1+\alpha)\frac{E_{k-1}-E_{l-1}}G\geq0.
\end{equation}
This proves the differential inequality.
If $\widetilde P$ is constant, then
\eqref{eq:Ptilde-Hessian} and \eqref{eq:Ftrace-gap}, together
with $V>0$, $G>0$ and $1+\alpha>0$, imply
$E_{k-1}=E_{l-1}$. Lemma~\ref{lem:quotient-algebra} gives $b=g$.
\end{proof}

\begin{proof}[Proof of Theorem~\ref{thm:star-quotient}]
Choose $\Phi$ and $V$ with respect to the point $o$ in the
statement. Lemma~\ref{lem:quotient-P} gives $P\leq P_0$ in
$\overline\Omega$, with equality on $\partial\Omega$.
Differentiation along an inward normal curve therefore gives
$P_\nu\geq0$ on the boundary.

Since $\nabla u=c_2\nu$ there, \eqref{eq:P-gradient} yields
\begin{equation}\label{eq:boundary-P-normal}
 P_\nu=2c_2(b_{\nu\nu}-1).
\end{equation}
The mixed Hessian terms vanish on the boundary. Indeed, for
every tangent vector $\tau$,
$$
 0=\tau(u_\nu)
 =\nabla^2u(\tau,\nu)
   +\langle\nabla u,\nabla_\tau\nu\rangle
 =\nabla^2u(\tau,\nu).
$$
Equation \eqref{eq:Ptilde-gradient} consequently gives
$$
 \widetilde P_\nu=-(b_{\nu\nu}-1)\Phi_\nu
 =-\frac{\Phi_\nu}{2c_2}P_\nu\leq0,
$$
where we used strict star-shapedness.

By Lemma~\ref{lem:quotient-Ptilde},
$F^{ij}\nabla_{ij}^2\widetilde P\geq0$.
If $\widetilde P$ were nonconstant, the strong maximum
principle would place a maximum at the boundary, and the
Hopf lemma would give $\widetilde P_\nu>0$ there
\cite[Chapter~3]{GilbargTrudinger}. This contradicts the
preceding inequality. Hence $\widetilde P$ is constant.
Lemma~\ref{lem:quotient-Ptilde} gives $b=g$.
The rigidity argument in the final part of
\cite[Section~5]{GaoMaYang} shows that $\Omega$ is a
geodesic ball and $u$ is radial with respect to its center.
\end{proof}

\section{An integral identity and symmetry without star-shapedness}
\label{sec:integral}

We derive a Rellich--Poho\v{z}aev identity for both equations.
For the quotient equation, we use the difference of two Newton
linearizations. For the sum equation, we use the linearization
of $G_k$. The resulting identity proves
Theorem~\ref{thm:nonstar-quotient} and, after a separate estimate
for the sum equation, Theorem~\ref{thm:sum}.

For the equation under consideration, set
\begin{equation}\label{eq:H-definition}
 H^{ij}=
 \begin{cases}
 G_k^{ij}-G_l^{ij},&\text{for \eqref{eq:quotient-problem}},\\
 G_k^{ij},&\text{for \eqref{eq:sum-problem}}.
 \end{cases}
\end{equation}
All these tensors are evaluated at $b=\nabla^2u+Kug$.

\begin{lemma}\label{lem:H-properties}
For either equation, $H^{ij}$ is positive definite and satisfies
\begin{equation}\label{eq:H-properties}
 \nabla_jH^{ij}=0,\qquad H^{ij}\nabla_mb_{ij}=0.
\end{equation}
As an endomorphism, $H$ commutes with $b$ and with $\nabla^2u$.
\end{lemma}

\begin{proof}
The divergence identity follows from
\eqref{eq:Newton-divfree}. For the quotient equation,
$G_k-G_l=0$ in $\Omega$, so
$$
 H^{ij}\nabla_mb_{ij}=\nabla_m(G_k-G_l)=0.
$$
For the sum equation, the same contraction is
$\nabla_mG_k=0$. The polynomial property of the Newton tensors
proved in Section~\ref{sec:preliminaries} gives commutation with
$b$. Since $\nabla^2u=b-Kug$, it also gives commutation with
$\nabla^2u$.

For the sum equation, positivity follows from
Lemma~\ref{lem:individual-ellipticity}. For the quotient equation,
$$
 H^{ij}=G F^{ij},
$$
where $G=G_k=G_l>0$. Proposition~\ref{prop:quotient-ellipticity}
then gives positivity.
\end{proof}

\begin{proposition}\label{prop:Pohozaev}
Let $u\in C^3(\overline\Omega)$ solve either
\eqref{eq:quotient-problem} or \eqref{eq:sum-problem}, and let
$H$ be given by \eqref{eq:H-definition}. Then
\begin{align}
 &\int_\Omega H^{ij}b_{ij}(u-Kc_1)V\,d\mu\notag\\
 &\qquad=
 \frac12\int_\Omega
 \bigl(|\nabla u|^2+Ku^2-K^3c_1^2-c_2^2\bigr)
 H^{ij}\delta_{ij}V\,d\mu.
 \label{eq:Pohozaev}
\end{align}
Equivalently, with $P$ and $P_0$ as in \eqref{eq:P-definition},
\begin{align}
 &\int_\Omega(Kc_1-u)
 \bigl(H^{ij}\delta_{ij}-H^{ij}b_{ij}\bigr)V\,d\mu\notag\\
 &\qquad=
 \frac12\int_\Omega(P-P_0)H^{ij}\delta_{ij}V\,d\mu.
 \label{eq:Pohozaev-P}
\end{align}
\end{proposition}

\begin{proof}
We follow the integration by parts used for the $k$-Hessian
equation in \cite[Section~6]{GaoMaYang}. The proof below uses
only the properties in Lemma~\ref{lem:H-properties}.
Since $\Phi_{lj}=V\delta_{lj}$ and
$b_{ij}=u_{ij}+Ku\delta_{ij}$,
\begin{align}
 H^{ij}b_{ij}uV
 &=\nabla_j(H^{ij}u_{il}u\Phi_l)
   -H^{ij}(\nabla_ju_{il})u\Phi_l\notag\\
 &\quad-H^{ij}u_{il}u_j\Phi_l
   +Ku^2VH^{ij}\delta_{ij}.
 \label{eq:Pohozaev-start}
\end{align}
The Codazzi identity and \eqref{eq:H-properties} imply
$$
 \begin{aligned}
 H^{ij}(\nabla_ju_{il})\Phi_l
 &=H^{ij}(\nabla_jb_{il}-Ku_j\delta_{il})\Phi_l\\
 &=H^{ij}\nabla_lb_{ij}\,\Phi_l-KH^{ij}u_j\Phi_i\\
 &=-KH^{ij}u_j\Phi_i.
 \end{aligned}
$$
Thus
\begin{equation}\label{eq:Pohozaev-u-square}
 -H^{ij}(\nabla_ju_{il})u\Phi_l
 =\frac K2\nabla_j(H^{ij}\Phi_i u^2)
  -\frac K2H^{ij}\delta_{ij}u^2V.
\end{equation}

Commutation of $H$ and $\nabla^2u$ gives
$$
 H^{ij}u_{il}u_j\Phi_l
 =H^{li}u_{ij}u_j\Phi_l
 =\frac12H^{li}\Phi_l\nabla_i|\nabla u|^2.
$$
Another application of \eqref{eq:H-properties} gives
\begin{equation}\label{eq:Pohozaev-gradient}
 -H^{ij}u_{il}u_j\Phi_l
 =-\frac12\nabla_i(H^{li}\Phi_l|\nabla u|^2)
  +\frac12H^{ij}\delta_{ij}|\nabla u|^2V.
\end{equation}
Substituting \eqref{eq:Pohozaev-u-square} and
\eqref{eq:Pohozaev-gradient} into \eqref{eq:Pohozaev-start},
and applying the divergence theorem, yields
\begin{align}
 \int_\Omega H^{ij}b_{ij}uV\,d\mu
 &=Kc_1\int_{\partial\Omega}H^{ij}u_{il}\Phi_l\nu_j\,dS
 \notag\\
 &\quad+\frac{K^3c_1^2-c_2^2}{2}
     \int_{\partial\Omega}H^{ij}\Phi_i\nu_j\,dS\notag\\
 &\quad+\frac12\int_\Omega(Ku^2+|\nabla u|^2)
          H^{ij}\delta_{ij}V\,d\mu.
 \label{eq:Pohozaev-boundary}
\end{align}
Here we used $u=Kc_1$ and $|\nabla u|^2=c_2^2$ on
$\partial\Omega$.

We now evaluate the two boundary integrals. First,
\begin{equation}\label{eq:boundary-Phi}
 \int_{\partial\Omega}H^{ij}\Phi_i\nu_j\,dS
 =\int_\Omega H^{ij}\Phi_{ij}\,d\mu
 =\int_\Omega H^{ij}\delta_{ij}V\,d\mu.
\end{equation}
Second, the Codazzi identity gives
$$
 \begin{aligned}
 \int_{\partial\Omega}H^{ij}b_{il}\Phi_l\nu_j\,dS
 &=\int_\Omega
    \left(H^{ij}\nabla_jb_{il}\,\Phi_l
          +H^{ij}b_{il}\Phi_{lj}\right)d\mu\\
 &=\int_\Omega
    \left(H^{ij}\nabla_lb_{ij}\,\Phi_l
          +H^{ij}b_{ij}V\right)d\mu\\
 &=\int_\Omega H^{ij}b_{ij}V\,d\mu.
 \end{aligned}
$$
On the boundary,
$u_{il}=b_{il}-K^2c_1\delta_{il}$. Therefore
\begin{equation}\label{eq:boundary-uHess}
 \int_{\partial\Omega}H^{ij}u_{il}\Phi_l\nu_j\,dS
 =\int_\Omega
   \bigl(H^{ij}b_{ij}-K^2c_1H^{ij}\delta_{ij}\bigr)V\,d\mu.
\end{equation}
Substitution into \eqref{eq:Pohozaev-boundary} gives
$$
 \begin{aligned}
 \int_\Omega H^{ij}b_{ij}uV\,d\mu
 &=Kc_1\int_\Omega H^{ij}b_{ij}V\,d\mu\\
 &\quad+\frac12\int_\Omega
  \bigl(|\nabla u|^2+Ku^2-K^3c_1^2-c_2^2\bigr)
  H^{ij}\delta_{ij}V\,d\mu.
 \end{aligned}
$$
This is \eqref{eq:Pohozaev}. Finally,
$$
 |\nabla u|^2+Ku^2-K^3c_1^2-c_2^2
 =P-P_0+2(u-Kc_1).
$$
Using this relation in \eqref{eq:Pohozaev} gives
\eqref{eq:Pohozaev-P}.
\end{proof}

\begin{proof}[Proof of Theorem~\ref{thm:nonstar-quotient}]
Lemma~\ref{lem:quotient-algebra} gives $E_1(b)\geq1$.
For $w=(u-Kc_1)/V$, identity \eqref{eq:positive-transform}
and the trace relation $\Delta u+nKu=nE_1(b)$ give
\begin{equation}\label{eq:nonstar-w}
 \Delta w+\frac2V\langle\nabla V,\nabla w\rangle
 =\frac{n(E_1-K^2c_1)}V\geq0.
\end{equation}
The boundary value of $w$ is zero. The function is not constant,
since $w_\nu=c_2/V>0$ on $\partial\Omega$.
The strong maximum principle therefore gives
\begin{equation}\label{eq:nonstar-u-less}
 u<Kc_1\quad\text{in }\Omega.
\end{equation}

For the quotient equation, $H^{ij}=G F^{ij}$.
By \eqref{eq:Ftrace-gap},
\begin{equation}\label{eq:quotient-H-gap}
 H^{ij}\delta_{ij}-H^{ij}b_{ij}
 \geq(1+\alpha)(E_{k-1}-E_{l-1})\geq0.
\end{equation}
Lemma~\ref{lem:quotient-P} gives $P\leq P_0$.
In \eqref{eq:Pohozaev-P}, the left-hand side is nonnegative
by \eqref{eq:nonstar-u-less} and \eqref{eq:quotient-H-gap}.
The right-hand side is nonpositive because
$H^{ij}\delta_{ij}>0$ and $V>0$. Both sides must vanish.
Since $Kc_1-u>0$ in $\Omega$, we obtain
$$
 H^{ij}\delta_{ij}-H^{ij}b_{ij}=0
 \quad\text{in }\Omega.
$$
Equation \eqref{eq:quotient-H-gap} now gives
$E_{k-1}=E_{l-1}$. Lemma~\ref{lem:quotient-algebra} implies
$b=g$. The rigidity argument in the final part of
\cite[Section~5]{GaoMaYang} shows that $\Omega$ is a
geodesic ball and $u$ is radial with respect to its center.
\end{proof}

For the sum equation, the estimates for the quotient do not
directly give the required differential inequality for $P$.
We prove it separately. In the rest of this section,
$u$ solves \eqref{eq:sum-problem}.

\begin{lemma}\label{lem:sum-P}
The function $P$ in \eqref{eq:P-definition} satisfies
\begin{equation}\label{eq:sum-LP}
 G_k^{ij}\nabla_{ij}^2P\geq0.
\end{equation}
Hence $P\leq P_0$ on $\overline\Omega$.
If $P$ is constant, then $b=g$.
For $k\geq2$, one also has
\begin{equation}\label{eq:sum-basic}
 E_1\geq1,\qquad E_{k-1}\geq1,\qquad E_{k-2}\geq1
\end{equation}
and
\begin{equation}\label{eq:sum-trace-gap}
 G_k^{ij}\delta_{ij}-G_k^{ij}b_{ij}
 \geq(k+\alpha)(E_{k-1}-1)\geq0.
\end{equation}
\end{lemma}

\begin{proof}
Lemma~\ref{lem:automatic-cone} gives
$b\in\widetilde\Gamma_k$ on $\overline\Omega$.
The operator $G_k^{ij}\nabla_{ij}^2$ is uniformly elliptic by
Lemma~\ref{lem:individual-ellipticity}.

If $k=1$, the equation is $E_1=1$ and
$G_1^{ij}=\delta_{ij}/n$. Contracting
\eqref{eq:P-Hessian}, and differentiating $E_1=1$, gives
$$
 \frac12G_1^{ij}\nabla_{ij}^2P
 =\frac1n|b|^2-1\geq0.
$$
The last inequality is the Cauchy--Schwarz inequality
$|b|^2\geq(\operatorname{tr}b)^2/n=n$.
Equality holds only when $b=g$. The maximum principle
gives the assertions in this case.

Assume now that $k\geq2$.
We first prove $E_{k-1}\geq1$.
Lemma~\ref{lem:cone-inclusion} gives $E_{k-1}>0$.
If $E_{k-1}<1$ and $E_k\leq0$, then
$E_k+\alpha E_{k-1}<1+\alpha$, contrary to the equation.
If $E_{k-1}<1$ and $E_k>0$, then $b\in\Gamma_k$, and
the Maclaurin inequality \eqref{eq:Maclaurin} gives
$$
 E_k\leq E_{k-1}^{k/(k-1)}<1.
$$
Again $G_k<1+\alpha$, a contradiction.
Thus $E_{k-1}\geq1$. Applying \eqref{eq:Maclaurin}
in $\Gamma_{k-1}$ gives $E_1\geq1$ and, if $k\geq3$,
$E_{k-2}\geq1$. For $k=2$, the last assertion is
$E_0=1$. This proves \eqref{eq:sum-basic}.

The equation also gives
\begin{equation}\label{eq:sum-Ek}
 E_k=1-\alpha(E_{k-1}-1).
\end{equation}
Using \eqref{eq:Euler-G}, \eqref{eq:trace-G} and
\eqref{eq:sum-Ek}, we obtain the exact identity
\begin{align}
 G_k^{ij}\delta_{ij}-G_k^{ij}b_{ij}
 &=k(E_{k-1}-E_k)
   +\alpha(k-1)(E_{k-2}-E_{k-1})\notag\\
 &=(k+\alpha)(E_{k-1}-1)
   +\alpha(k-1)(E_{k-2}-1)\geq0.
 \label{eq:sum-trace-exact}
\end{align}
This proves \eqref{eq:sum-trace-gap}.

We next estimate the quadratic contraction.
Equations \eqref{eq:Euler-G}, \eqref{eq:square-G} and
\eqref{eq:sum-Ek} give
\begin{align}
 &G_k^{ij}b_{mi}b_{mj}-G_k^{ij}b_{ij}\notag\\
 &=n(1+\alpha)E_1-(n-k)E_{k+1}
   -\alpha(n-k+1)E_k-k(1+\alpha)+\alpha E_{k-1}\notag\\
 &=n(1+\alpha)(E_1-1)+(n-k)(1-E_{k+1})\notag\\
 &\quad+
   \bigl(\alpha^2(n-k+1)+\alpha\bigr)(E_{k-1}-1).
 \label{eq:sum-square-first}
\end{align}
Newton's inequality \eqref{eq:Newton} gives
$$
 E_{k+1}\leq\frac{E_k^2}{E_{k-1}}.
$$
For $k=n$, we use $E_{n+1}=0$ and the terms multiplied
by $n-k$ vanish. Notice that this use of Newton's inequality
does not require $E_k>0$.

To separate the nonnegative terms in
\eqref{eq:sum-square-first}, write
$$
 1-E_{k+1}
 =1-\frac{E_k^2}{E_{k-1}}
  +\frac{E_k^2}{E_{k-1}}-E_{k+1}.
$$
Equation \eqref{eq:sum-Ek} gives
$$
 1-\frac{E_k^2}{E_{k-1}}
 =(E_{k-1}-1)
   \left(\frac{(1+\alpha)^2}{E_{k-1}}-\alpha^2\right).
$$
Consequently,
$$
 \begin{aligned}
 &(n-k)\left(1-\frac{E_k^2}{E_{k-1}}\right)
   +\bigl(\alpha^2(n-k+1)+\alpha\bigr)(E_{k-1}-1)\\
 &\qquad=
 (1+\alpha)
 \left(\alpha+\frac{(n-k)(1+\alpha)}{E_{k-1}}\right)
 (E_{k-1}-1).
 \end{aligned}
$$
Substitution into \eqref{eq:sum-square-first} now yields
\begin{align}
 &G_k^{ij}b_{mi}b_{mj}-G_k^{ij}b_{ij}\notag\\
 &=n(1+\alpha)(E_1-1)
   +(n-k)\left(\frac{E_k^2}{E_{k-1}}-E_{k+1}\right)\notag\\
 &\quad+
 (1+\alpha)
 \left(\alpha+\frac{(n-k)(1+\alpha)}{E_{k-1}}\right)
 (E_{k-1}-1)\geq0.
 \label{eq:sum-square-gap}
\end{align}
Each term on the right is nonnegative by
\eqref{eq:sum-basic} and Newton's inequality.

Since $G_k$ is constant, contraction of
\eqref{eq:P-Hessian} gives
\begin{align}
 \frac12G_k^{ij}\nabla_{ij}^2P
 &=G_k^{ij}b_{mi}b_{mj}-G_k^{ij}b_{ij}\notag\\
 &\quad+Ku\bigl(G_k^{ij}\delta_{ij}-G_k^{ij}b_{ij}\bigr).
 \label{eq:sum-P-decomposition}
\end{align}
Equations \eqref{eq:sum-trace-exact} and
\eqref{eq:sum-square-gap}, together with $Ku\geq0$,
prove \eqref{eq:sum-LP}. The maximum principle and
$P=P_0$ on $\partial\Omega$ give $P\leq P_0$.

Suppose that $P$ is constant. Both terms on the right-hand
side of \eqref{eq:sum-P-decomposition} are nonnegative.
The first must therefore vanish. Its decomposition in
\eqref{eq:sum-square-gap} implies $E_1=1$. On the other hand,
\eqref{eq:sum-basic} and the Maclaurin inequality in
$\Gamma_{k-1}$ give
$$
 1\leq E_{k-1}\leq E_1^{k-1}=1.
$$
Hence $E_{k-1}=1$, and \eqref{eq:sum-Ek} gives $E_k=1$.
We now have $b\in\Gamma_k$ and equality in
$E_k^{1/k}\leq E_1$. The equality case in
\eqref{eq:Maclaurin} gives $b=g$.
\end{proof}

\begin{proof}[Proof of Theorem~\ref{thm:sum}]
We first consider $k=1$. In this case,
$$
 E_1=1,\qquad
 H^{ij}=\frac{\delta_{ij}}n,\qquad
 H^{ij}b_{ij}=H^{ij}\delta_{ij}=1.
$$
Identity \eqref{eq:Pohozaev-P} reduces to
$$
 \int_\Omega(P-P_0)V\,d\mu=0.
$$
Lemma~\ref{lem:sum-P} gives $P\leq P_0$, and $V>0$.
Therefore $P=P_0$ throughout $\Omega$. The same lemma
gives $b=g$. The rigidity argument in the final part of
\cite[Section~5]{GaoMaYang} shows that $\Omega$ is a
geodesic ball and $u$ is radial with respect to its center.
This argument does not use $K^2c_1\leq1$.

Assume that $k\geq2$. By \eqref{eq:sum-basic},
$E_1\geq1$. For $w=(u-Kc_1)/V$, we have
$$
 \Delta w+\frac2V\langle\nabla V,\nabla w\rangle
 =\frac{n(E_1-K^2c_1)}V\geq0.
$$
Since $w=0$ and $w_\nu=c_2/V>0$ on $\partial\Omega$,
the strong maximum principle gives $u<Kc_1$ in $\Omega$.
In \eqref{eq:Pohozaev-P}, we now have $H^{ij}=G_k^{ij}$.
Its left-hand side is nonnegative by
\eqref{eq:sum-trace-gap}. Its right-hand side is
nonpositive by Lemma~\ref{lem:sum-P}.
Thus
$$
 \int_\Omega(P-P_0)G_k^{ij}\delta_{ij}V\,d\mu=0.
$$
The weight is strictly positive, so $P=P_0$ in $\Omega$.
Lemma~\ref{lem:sum-P} yields $b=g$.
The rigidity argument in the final part of
\cite[Section~5]{GaoMaYang} shows that $\Omega$ is a
geodesic ball and $u$ is radial with respect to its center.
\end{proof}

\par\noindent{\large\bf Acknowledgements}\par
The authors would like to thank  Professor Jiabin Yin for his helpful discussions and suggestions on this work. This work is supported NSFC grant No. 12571056 and NSF of Henan Province No. 252300421497 and No. 262300421869.

  \vskip4mm
\par\noindent{\large\bf Conflict of interest}\par
The authors declare that they have no conflict of interest.

\end{document}